\documentclass[12pt]{amsart}

\usepackage{amssymb,latexsym}
\usepackage{amsthm,amsfonts}
\usepackage[centertags]{amsmath}
\usepackage[a4paper,margin=22mm]{geometry}

\newcommand{\N}{{\mathbb{N}}}

\newcommand{\Z}{{\mathbb{Z}}}

\newcommand{\F}{{\mathbb{F}}}

\makeatletter
\ifcase\@ptsize

\or

\or

\fi
\makeatother

\newcommand{\bse}{{\bf e}}
\newcommand{\x}{{\bf x}}

\newcommand{\cP}{{\mathcal{P}}}

\newcommand{\cL}{{\mathcal{L}}}
\newcommand{\cS}{{\mathcal{S}}}
\newcommand{\se}{sequence }
\newcommand{\ses}{sequences }
\newcommand{\dis}{discrepancy }
\newcommand{\con}{construction }
\newcommand{\gff}{global function field }
\newcommand{\gffs}{global function fields }
\newcommand{\qmc}{quasi-Monte Carlo }
\newcommand{\me}{method }
\newcommand{\mes}{methods }

\newcommand{\fif}{finite field }

\newcommand{\abs}[1]{\left| #1 \right|}

\renewcommand{\le}{\leqslant}
\renewcommand{\ge}{\geqslant}

\newtheorem{theorem}{Theorem}
\newtheorem{proposition}{Proposition}
\newtheorem{corollary}{Corollary}
\newtheorem{lemma}{Lemma}

\newtheorem{definition}{Definition}

\newtheorem{remark}{Remark}

\begin{document}
\title
[Halton-type sequences]{Halton-type sequences from global function fields}

\author{HARALD NIEDERREITER and ANDERSON SIANG JING YEO}
\address{Harald Niederreiter,
Johann Radon Institute for Computational and Applied Mathematics, Austrian Academy of Sciences, Altenbergerstr. 69, A-4040 Linz,
Austria, and Department of Mathematics, University of Salzburg, Hellbrunnerstr. 34, A-5020 Salzburg, Austria;
Anderson Siang Jing Yeo, School of Physical and Mathematical Sciences, Nanyang Technological University, 21 Nanyang Link, 
NTU SPMS-MAS-04-20, Singapore 637371, Republic of Singapore}
\email{ghnied@gmail.com (Harald Niederreiter), yeos0041@e.ntu.edu.sg (Anderson Siang Jing Yeo)}


\date{\today}



\begin{abstract}
For any prime power $q$ and any dimension $s \ge 1$, a new construction of $(t,s)$-sequences in base $q$ using global function
fields is presented. The construction yields an analog of Halton sequences for global function fields. It is the first
general construction of $(t,s)$-sequences that is not based on the digital method. The construction can also be put into the
framework of the theory of $(u,\bse,s)$-sequences that was recently introduced by Tezuka and leads in this way to better
discrepancy bounds for the constructed sequences.
\end{abstract}

\maketitle

\section{Introduction} \label{se1}

The \con of low-\dis \ses is an important problem in number theory and combinatorics and it is also highly relevant for \qmc \mes
in scientific computing (see \cite{DP, N92}). Let us recall the definition of the star \dis and the associated definition of a low-\dis
sequence. Let $s \ge 1$ be a given dimension and let $\cP$ be a point set consisting of the $N$ points $\x_0,\ldots,\x_{N-1}$ in
the $s$-dimensional unit cube $[0,1]^s$. For a subinterval $J$ of $[0,1]^s$, let $A(J;\cP)$ be the number of integers $n$ with
$0 \le n \le N-1$ and $\x_n \in J$. Then the \emph{star \dis} $D_N^*(\cP)$ of $\cP$ is defined by
$$
D_N^*(\cP)=\sup_J \abs{\frac{A(J;\cP)}{N} - \lambda_s(J)},
$$
where the supremum is extended over all subintervals $J$ of $[0,1]^s$ with one vertex at the origin and $\lambda_s$ denotes the
$s$-dimensional Lebesgue measure. For a sequence $\cS$ of points $\x_0,\x_1,\ldots$ in $[0,1]^s$, we define the star \dis $D_N^*(\cS)$
for all $N \ge 1$ by putting $D_N^*(\cS)=D_N^*(\cP_N)$, where $\cP_N$ is the point set consisting of the first $N$ terms $\x_0,\ldots,
\x_{N-1}$ of $\cS$. We say that $\cS$ is a \emph{low-\dis \se} if
\begin{equation} \label{eqdi}
D_N^*(\cS)=O(N^{-1} (\log N)^s) \qquad \mbox{for all } N \ge 2,
\end{equation}
where the implied constant is independent of $N$. This is the smallest possible order of magnitude that can currently be obtained for
the star \dis of a \se of points in $[0,1]^s$. Actually, it is a widely held belief that no smaller order of magnitude can be achieved.

Historically, the first \con of low-\dis \ses for any dimension was that of Halton \ses in~\cite{Ha}. For any integer
$b \ge 2$, let $Z_b=\{0,1,\ldots, b-1\} \subset \Z$ denote the least residue system modulo $b$. Every integer $n \ge 0$ has a unique
digit expansion
$$
n=\sum_{r=0}^{\infty} a_r(n)b^r
$$
in base $b$, where $a_r(n) \in Z_b$ for all $r \ge 0$ and $a_r(n)=0$ for all sufficiently large $r$. The \emph{radical-inverse function}
$\phi_b$ in base $b$ is defined by
$$
\phi_b(n)=\sum_{r=0}^{\infty} a_r(n)b^{-r-1} \in [0,1) \qquad \mbox{for all } n \in \N_0.
$$
As usual, we write $\N_0$ for the set of nonnegative integers and $\N$ for the set of positive integers. Now choose pairwise coprime
integers $b_1,\ldots,b_s \ge 2$ and define the \emph{Halton \se in the bases} $b_1,\ldots,b_s$ as the \se ${\bf h}_0,{\bf h}_1,\ldots$ with
$$
{\bf h}_n=(\phi_{b_1}(n),\ldots,\phi_{b_s}(n)) \in [0,1)^s \qquad \mbox{for } n=0,1,\ldots .
$$
An explicit upper bound of the form~\eqref{eqdi} for the star \dis of this Halton \se can be found in \cite[Theorem 3.36]{DP}. In
practical implementations, the bases $b_1,\ldots,b_s$ are usually taken to be distinct prime numbers.

Modern \mes for the \con of $s$-dimensional low-\dis \ses are based on the theory of $(t,m,s)$-nets and $(t,s)$-\ses which was developed
in~\cite{N87}. This theory has a strong combinatorial flavor; see \cite[Chapter~6]{DP}, \cite[Chapter~4]{N92}, and the article on
$(t,m,s)$-nets in the handbook~\cite{CD}. The currently most powerful constructions of $(t,s)$-\ses use \gffs and quite a number of
such constructions are available; see \cite{HN, MN, NO, NX95, NX96, XN} as well as the surveys in \cite{DP, NX96a, NX98, NX01}.

In this paper we present a new \con of low-\dis \ses using global function fields. The \con is comparatively simple and inspired by the
\con of Halton sequences. It fits very well into the framework of the theory of $(u,\bse,s)$-\ses developed recently by
Tezuka~\cite{T12}. This theory is an extension of the theory of $(t,s)$-\ses and leads in many cases to better upper bounds on the star
discrepancy. The new \con is exceptional in the sense that it is the first general \con of $(t,s)$-\ses and $(u,\bse,s)$-\ses that is
not based on the digital method, which so far was the standard \me of constructing $(t,s)$-\ses and $(u,\bse,s)$-\ses (see
\cite[Chapter~4]{DP} for a description of the digital method). Rather, our \con provides a direct and explicit formula for the terms of
the \se (see equation~\eqref{eqc} in Section~\ref{se3}).

The rest of the paper is organized as follows. In Section~\ref{se2} we present the necessary background on $(u,\bse,s)$-sequences.
Section~\ref{se3} contains the \con of our Halton-type sequences. In Section~\ref{se4} we show that these \ses are both $(u,\bse,s)$-\ses
and $(t,s)$-\ses with suitable parameters $u$, $\bse$, and $t$. Upper bounds on the star \dis of these \ses are discussed in
Section~\ref{se5}.

\section{Background on $(u,\bse,s)$-sequences} \label{se2}

Tezuka~\cite{T12} recently introduced the concepts of $(u,m,\bse,s)$-nets and $(u,\bse,s)$-sequences. For several reasons, e.g., to
prove Proposition~\ref{pr1} below and to avoid unnecessary additional conditions, his definitions were slightly revised by Hofer
and Niederreiter~\cite{HN}. We follow this revised approach.

\begin{definition} \label{de1} {\rm
Let $b\ge 2$, $s\ge 1$, and $0\le u \le m$ be integers and let $\bse=(e_1,\ldots, e_s) \in \N^s$ be an $s$-tuple of positive integers.
A \emph{$(u,m,\bse,s)$-net in base $b$} is a point set $\cP$ of $b^m$ points in $[0,1)^s$ such that $A(J;\cP)=b^m \lambda_s(J)$ for every
interval $J$ of the form
\begin{equation} \label{eqi}
J=\prod_{i=1}^s\left[a_ib^{-d_i},(a_i+1)b^{-d_i}\right)
\end{equation}
with integers $d_i\ge 0$, $0 \le a_i < b^{d_i}$, and $e_i|d_i$ for $1\le i\le s$ and with $\lambda_s(J)\ge b^{u-m}$.}
\end{definition}

The classical concept of a $(u,m,s)$-net in base $b$ corresponds to the special case $\bse =(1,\ldots,1)$ in Definition~\ref{de1}.
For a real number $x \in [0,1]$, let
$$
x=\sum_{j=1}^{\infty} y_jb^{-j} \qquad \mbox{with all } y_j \in Z_b
$$
be a $b$-adic expansion of $x$, where the case $y_j=b-1$ for all sufficiently large $j$ is allowed. For any integer
$m \ge 1$, we define the truncation
$$
[x]_{b,m} = \sum_{j=1}^m y_jb^{-j}.
$$
If $\x =(x^{(1)},\ldots,x^{(s)}) \in [0,1]^s$, then the truncation $[\x]_{b,m}$ is defined coordinatewise, that is,
$$
[\x]_{b,m} = \left([x^{(1)}]_{b,m},\ldots,[x^{(s)}]_{b,m} \right).
$$

\begin{definition} \label{de2} {\rm
Let $b \ge 2$, $s \ge 1$, and $u \ge 0$ be integers and let $\bse \in \N^s$.
A sequence $\x_0,\x_1,\ldots$ of points in $[0,1]^s$ is called a \emph{$(u,\bse,s)$-sequence in base $b$} if for all
integers $k\ge 0$ and $m>u$
the points $[\x_n]_{b,m}$ with $kb^m\le n<(k+1)b^m$ form a $(u,m,\bse,s)$-net in base $b$.}
\end{definition}

The classical concept of a $(u,s)$-\se in base $b$ corresponds to the special case $\bse =(1,\ldots,1)$ in Definition~\ref{de2}.
The following result was shown in~\cite{HN}.

\begin{proposition} \label{pr1}
Let $b\ge 2$, $s\ge 1$, and $u\ge 0$ be integers and let $\bse=(e_1,\ldots,e_s)\in \N^s$.
Then any $(u,\bse,s)$-sequence in base $b$ is a $(t,s)$-sequence in base $b$ with
$$
t=u+\sum_{i=1}^s(e_i-1).
$$
\end{proposition}

For $\bse \ne (1,\ldots,1)$, there are currently four families of \ses for which it has been verified that they are $(u,\bse,s)$-sequences:
the Niederreiter \ses constructed in~\cite{N88}, the generalized Niederreiter \ses constructed in~\cite{T93}, the Niederreiter-Xing
\ses constructed in~\cite{XN}, and the Hofer-Niederreiter \ses constructed in~\cite{HN}. The proof of the property of being
$(u,\bse,s)$-\ses is given in~\cite{T12} for Niederreiter \ses and generalized Niederreiter \ses and in~\cite{HN} for Niederreiter-Xing
\ses and Hofer-Niederreiter sequences.

\section{The construction} \label{se3}

As we have already mentioned, our new \con of low-\dis \ses is based on global function fields. Recall that a \emph{\gff} is an
algebraic function field of one variable over a finite field. We follow the monograph~\cite{St} regarding terminology and notation
for global function fields.

We write $\F_q$ for the finite field with $q$ elements, where $q$ is an arbitrary prime power. Let $F$ be a \gff with full constant
field $\F_q$, that is, with $\F_q$ algebraically closed in $F$. We assume that $F$ has at least one rational place, that is, a place of
degree $1$. Given a dimension $s \ge 1$, we choose $s+1$ distinct places $P_{\infty},P_1,\ldots,P_s$ of $F$ with
$\deg(P_{\infty})=1$. The degrees of the places $P_1,\ldots,P_s$ are arbitrary and we put $e_i=\deg(P_i)$ for $1 \le i \le s$.
Denote by $O_F$ the holomorphy ring given by
$$
O_F=\bigcap_{P \ne P_{\infty}} O_P,
$$
where the intersection is extended over all places $P \ne P_{\infty}$ of $F$ and $O_P$ is the valuation ring of $P$.
For each $i=1,\ldots,s$, let $\wp_i$ be the maximal ideal of $O_F$ corresponding to $P_i$. Then the residue class field
$F_{P_i} := O_F/\wp_i$ has order $q^{e_i}$ (see \cite[Proposition 3.2.9]{St}). We fix a bijection
\begin{equation} \label{eqsp}
\sigma_{P_i}: F_{P_i} \to Z_{q^{e_i}}.
\end{equation}

For any divisor $D$ of $F$, we write $\cL(D)$ for the Riemann-Roch space associated to $D$ and $\ell(D)$ for the
dimension of the vector space $\cL(D)$ over $\F_q$ (see \cite[Section~1.4]{St}). For each $i=1,\ldots,s$, we can obtain a
local parameter $z_i \in O_F$ at $\wp_i$, by applying the Riemann-Roch theorem \cite[Theorem 1.5.17]{St} and choosing
$$
z_i \in \cL(kP_{\infty} -P_i) \setminus \cL(kP_{\infty} -2P_i)
$$
for a suitably large integer $k$. Then for every $f \in O_F$, we have a local expansion of $f$ at $\wp_i$ of the form
\begin{equation} \label{eqe}
f=\sum_{j=0}^{\infty} f_{i,j} z_i^j \qquad \mbox{with all } f_{i,j} \in F_{P_i}.
\end{equation}
Compare with \cite[pp. 5--6]{NX01} for a detailed description of how to obtain a local expansion.
Now we define the map $\phi: O_F \to [0,1]^s$ by
\begin{equation} \label{eqp}
\phi(f)=\left(\sum_{j=0}^{\infty} \sigma_{P_1}(f_{1,j}) (q^{e_1})^{-j-1},\ldots,\sum_{j=0}^{\infty} \sigma_{P_s}(f_{s,j})
(q^{e_s})^{-j-1} \right) \qquad \mbox{for all } f \in O_F.
\end{equation}
The map $\phi$ can be viewed as an $s$-dimensional radical-inverse function for $O_F$.

We arrange the elements of $O_F$ into a sequence by using the fact that
$$
O_F = \bigcup_{m=0}^{\infty} \cL(mP_{\infty}).
$$
The terms of this \se are denoted by $f_0,f_1,\ldots$ and they are obtained as follows. Consider the chain
$$
\cL(0) \subseteq \cL(P_{\infty}) \subseteq \cL(2P_{\infty}) \subseteq \cdots
$$
of vector spaces over $\F_q$. At each step of this chain, the dimension either remains the same or increases by $1$.
From a certain point on, the dimension always increases by $1$ according to the Riemann-Roch theorem. Thus, we can
construct a sequence $v_0,v_1,\ldots$ of elements of $O_F$ such that for each $m \in \N_0$,
$$
\{v_0,v_1,\ldots,v_{\ell(mP_{\infty})-1} \}
$$
is an $\F_q$-basis of $\cL(mP_{\infty})$. For $n \in \N_0$, let
$$
n=\sum_{r=0}^{\infty} a_r(n)q^r \qquad \mbox{with all } a_r(n) \in Z_q
$$
be the digit expansion of $n$ in base $q$. Note that $a_r(n)=0$ for all sufficiently large $r$. We fix a bijection
$\eta: Z_q \to \F_q$ with $\eta(0)=0$. Then we define
\begin{equation} \label{eqf}
f_n=\sum_{r=0}^{\infty} \eta(a_r(n))v_r \in O_F \qquad \mbox{for } n=0,1,\ldots .
\end{equation}
Note that the sum above is finite since for each $n \in \N_0$ we have $\eta(a_r(n))=0$ for all sufficiently large $r$.

Using~\eqref{eqp} and~\eqref{eqf}, we now define the sequence $\x_0,\x_1,\ldots$ of points in $[0,1]^s$ by
\begin{equation} \label{eqc}
\x_n = \phi(f_n) \qquad \mbox{for } n=0,1,\ldots .
\end{equation}
This is our Halton-type sequence obtained from the global function field $F$.
We will show that this sequence is a $(u,\bse,s)$-sequence in base $q$ for some $u$ and $\bse$ (see Theorem~\ref{th1} below)
and also a $(t,s)$-\se in base $q$ for some $t$ (see Corollary~\ref{co1} below).

\begin{remark} \label{re1} {\rm
The \se $\x_0,\x_1,\ldots$ is obtained by the direct and explicit formula~\eqref{eqc} and it is easy to see that, in general, it
cannot be produced by the digital \me for the \con of $(t,s)$-sequences. The simple reason is that the digital \me for the \con of
$(t,s)$-\ses in base $q$ is a linear-algebra technique which works entirely in the \fif $\F_q$ (compare again with \cite[Chapter~4]{DP}),
whereas our \con uses extension fields of $\F_q$ of degrees $e_1,\ldots,e_s$. Moreover, whenever $e_i \ge 2$ there is no reasonable
way in which the bijection $\sigma_{P_i}$ in~\eqref{eqsp} can be interpreted as an $\F_q$-linear map. There is only one situation
in which our \con can be put into the framework of the digital method, namely when $q$ is a prime number and $e_i=1$ for $1 \le i \le s$.
In this case the bijections $\sigma_{P_i}$ in~\eqref{eqsp} can be chosen as identity maps. }
\end{remark}

\begin{remark} \label{re2} {\rm
Tezuka~\cite{T93} introduced a \con of Halton-type \ses based on rational function fields over finite fields. If we specialize our
\con to the case where $F$ is the rational function field over $\F_q$, then our \con is in general different from Tezuka's
construction. One significant difference is that Tezuka's \con can be described in terms of the digital method, whereas our \con
cannot in general be put into the framework of the digital \me (see Remark~\ref{re1}). }
\end{remark}

\section{The main result} \label{se4}

Before we formulate our main result in Theorem~\ref{th1} below, we prove the following lemma. We keep the notation from the
previous sections.

\begin{lemma} \label{le1}
Let $g$ be the genus of $F$ and let $k \ge 0$ and $m \ge g$ be integers. Let the elements $f_n \in O_F$ be defined by~\eqref{eqf}.
Then for $f \in O_F$ we have $f=f_n$ for some
integer $n$ with $kq^m \le n < (k+1)q^m$ if and only if $f=h+c$ with $h \in O_F$ depending only on $k$ and $m$ and with
$c \in \cL((m+g-1)P_{\infty})$.
\end{lemma}

\begin{proof}
The integers $n$ with $kq^m \le n < (k+1)q^m$ have digit expansions in base $q$ of the form
$$
n=\sum_{r=0}^{m-1} a_r(n)q^r + \sum_{r=m}^{\infty} a_r(n)q^r,
$$
where the $a_r(n)$ with $0 \le r \le m-1$ range independently over $Z_q$ and the $a_r(n)$ with $r \ge m$ depend
only on $k$. Thus, the corresponding $f_n$ have the form
\begin{equation} \label{eqs}
f_n=\sum_{r=0}^{m-1} \eta(a_r(n))v_r + \sum_{r=m}^{\infty} \eta(a_r(n))v_r =: \sum_{r=0}^{m-1} \alpha_r v_r +h,
\end{equation}
where the $\alpha_r =\eta(a_r(n))$ with $0 \le r \le m-1$ range independently over $\F_q$ and $h \in O_F$ depends only
on $k$ and $m$.

Now we consider $\cL((m+g-1)P_{\infty})$ and note that $\deg((m+g-1)P_{\infty})=m+g-1 \ge 2g-1$. Therefore
$\ell((m+g-1)P_{\infty})=m$ by the Riemann-Roch theorem. It follows that $\{v_0,v_1,\ldots,v_{m-1}\}$ is an $\F_q$-basis
of $\cL((m+g-1)P_{\infty})$, and so the linear combinations $\sum_{r=0}^{m-1} \alpha_r v_r$ range exactly over
$\cL((m+g-1)P_{\infty})$. In view of~\eqref{eqs}, the proof is complete.
\end{proof}

\begin{theorem} \label{th1}
The sequence $\x_0,\x_1,\ldots$ given by~\eqref{eqc} is a $(u,\bse,s)$-sequence in base $q$ with $u=g$ and $\bse =(e_1,\ldots,e_s)$,
where $g$ is the genus of $F$ and $e_i=\deg(P_i)$ for $1 \le i \le s$.
\end{theorem}

\begin{proof}
We proceed by Definition~\ref{de2} and fix integers $k \ge 0$ and $m > g$.
We have to show that the point set $\cP_{k,m}$ consisting of the points $[\x_n]_{q,m} = [\phi(f_n)]_{q,m}$ with $kq^m \le n < (k+1)q^m$
forms a $(g,m,\bse,s)$-net in base $q$.

Let $J$ be an interval of the form~\eqref{eqi} with $\lambda_s(J) \ge q^{g-m}$, that is,
\begin{equation} \label{eqd}
\sum_{i=1}^s d_i \le m-g.
\end{equation}
Note that we assume $e_i|d_i$ for $1 \le i \le s$ according to Definition~\ref{de1}.
Let $f \in O_F$ with local expansions as in~\eqref{eqe}. Then $[\phi(f)]_{q,m} \in J$ if and only if
$$
\left[\sum_{j=1}^{\infty} \sigma_{P_i}(f_{i,j-1}) (q^{e_i})^{-j} \right]_{q,m} \in [a_iq^{-d_i},(a_i+1)q^{-d_i} )
\qquad \mbox{for } 1 \le i \le s.
$$
This is equivalent to
$$
w_i := \sum_{j=1}^{d_i/e_i} \sigma_{P_i}(f_{i,j-1}) q^{-je_i} \in [a_iq^{-d_i},(a_i+1)q^{-d_i} )
\qquad \mbox{for } 1 \le i \le s.
$$
This condition is satisfied if and only if, for $1 \le i \le s$, the first $d_i$ $q$-adic digits of $w_i$ agree with
the first $d_i$ $q$-adic digits of $a_iq^{-d_i}$ and thus have fixed values depending only on $J$.
Note that these $d_i$ fixed $q$-adic digits of $w_i$ correspond to $d_i/e_i$ fixed $q^{e_i}$-adic digits. Therefore
for some $a_{i,j} \in Z_{q^{e_i}}$, $1 \le j \le d_i/e_i$, $1 \le i \le s$, we get the condition
$$
\sigma_{P_i}(f_{i,j-1})=a_{i,j} \qquad \mbox{for } 1 \le j \le d_i/e_i, \ 1 \le i \le s,
$$
or the equivalent condition
\begin{equation} \label{eqb}
f_{i,j-1} = \sigma_{P_i}^{-1}(a_{i,j}) =: \beta_{i,j} \in F_{P_i} \qquad \mbox{for } 1 \le j \le d_i/e_i, \ 1 \le i \le s.
\end{equation}
Now we recall~\eqref{eqe} and put
$$
h_i=\sum_{j=0}^{(d_i/e_i)-1} \beta_{i,j+1} z_i^j \in O_F \qquad \mbox{for } 1 \le i \le s.
$$
Then~\eqref{eqb} is equivalent to
\begin{equation} \label{eqa}
f \equiv h_i \ {\rm mod} \ \wp_i^{d_i/e_i} \qquad \mbox{for } 1 \le i \le s.
\end{equation}
Note that the elements $h_1,\ldots,h_s$ depend only on $J$.

We consider now the number $A(J;\cP_{k,m})$ of integers $n$ with $kq^m \le n < (k+1)q^m$ such that $[\phi(f_n)]_{q,m} \in J$. As we have just seen,
the latter condition is equivalent to~\eqref{eqa} holding for $f=f_n$. By Lemma~\ref{le1} the condition $kq^m \le n < (k+1)q^m$ is
equivalent to $f_n=h+c$ with fixed $h \in O_F$ and some $c \in \cL((m+g-1)P_{\infty})$. Thus, both conditions hold
simultaneously if and only if $c \in \cL((m+g-1)P_{\infty})$ satisfies
\begin{equation} \label{eqh}
c \equiv h_i-h \ {\rm mod} \ \wp_i^{d_i/e_i} \qquad \mbox{for } 1 \le i \le s.
\end{equation}
Consider the map
$$
\psi: \cL((m+g-1)P_{\infty}) \to R := (O_F/\wp_1^{d_1/e_1}) \times \cdots \times (O_F/\wp_s^{d_s/e_s})
$$
defined by
$$
\psi(c)=(c \ {\rm mod} \ \wp_1^{d_1/e_1},\ldots,c \ {\rm mod} \ \wp_s^{d_s/e_s}) \qquad \mbox{for } c \in \cL((m+g-1)P_{\infty}).
$$
Note that $\psi$ is a linear transformation between vector spaces over $\F_q$. We claim that $\psi$ is surjective. To prove this, it
suffices to show that
$$
\dim(\cL((m+g-1)P_{\infty})/{\rm ker}(\psi)) = \dim(R).
$$
The dimension $\dim(R)$ is easily computed to be
$$
\dim(R)=\sum_{i=1}^s \dim(O_F/\wp_i^{d_i/e_i})=\sum_{i=1}^s d_i.
$$
From the definition of $\psi$ it is clear that
\begin{equation} \label{eqk}
{\rm ker}(\psi)=\cL \Big((m+g-1)P_{\infty} -\sum_{i=1}^s (d_i/e_i)P_i \Big).
\end{equation}
Furthermore, by taking into account~\eqref{eqd} we get
\begin{equation} \label{eqr}
\deg \Big((m+g-1)P_{\infty} -\sum_{i=1}^s (d_i/e_i)P_i \Big) = m+g-1-\sum_{i=1}^s d_i \ge 2g-1.
\end{equation}
Therefore
\begin{eqnarray*}
\dim(\cL((m+g-1)P_{\infty})/{\rm ker}(\psi)) & = & \dim \Big(\cL((m+g-1)P_{\infty})/\cL \Big((m+g-1)P_{\infty}
-\sum_{i=1}^s (d_i/e_i)P_i \Big) \Big) \\
& = & \deg \Big(\sum_{i=1}^s (d_i/e_i)P_i \Big) = \sum_{i=1}^s d_i
\end{eqnarray*}
by the Riemann-Roch theorem, and so $\psi$ is indeed surjective.

Since $\psi$ is surjective, the number of $c \in \cL((m+g-1)P_{\infty})$ satisfying the system of congruences
in~\eqref{eqh} is equal to $\# \, {\rm ker}(\psi)$. By~\eqref{eqk}, \eqref{eqr}, and the Riemann-Roch theorem we get
\begin{eqnarray*}
\dim({\rm ker}(\psi)) & = & \ell \Big((m+g-1)P_{\infty} -\sum_{i=1}^s (d_i/e_i)P_i \Big) \\
& = & \deg \Big((m+g-1)P_{\infty} -\sum_{i=1}^s (d_i/e_i)P_i \Big) +1-g = m - \sum_{i=1}^s d_i.
\end{eqnarray*}
Thus, we have shown that
$$
A(J;\cP_{k,m})= \# \, {\rm ker}(\psi)=q^{m-\sum_{i=1}^s d_i}.
$$
On the other hand, we have
$$
q^m \lambda_s(J)=q^{m-\sum_{i=1}^s d_i}
$$
by the form of $J$, and so $\cP_{k,m}$ is indeed a $(g,m,\bse,s)$-net in base $q$ according to Definition~\ref{de1}.
\end{proof}

\begin{corollary} \label{co1}
The \se $\x_0,\x_1,\ldots$ given by~\eqref{eqc} is a $(t,s)$-\se in base $q$ with
$$
t=g+\sum_{i=1}^s (e_i-1),
$$
where $g$ is the genus of $F$ and $e_i=\deg(P_i)$ for $1 \le i \le s$.
\end{corollary}

\begin{proof}
This follows from Theorem~\ref{th1} and Proposition~\ref{pr1}.
\end{proof}

\begin{remark} \label{re3} {\rm
The Niederreiter-Xing \ses constructed in~\cite{XN} use the same ingredients as our construction, to wit a \gff $F$ with full
constant field $\F_q$ and $s+1$ distinct places $P_{\infty},P_1,\ldots,P_s$ of $F$ with $\deg(P_{\infty})=1$ and the degrees
$e_i=\deg(P_i)$ for $1 \le i \le s$ being arbitrary. It was shown in~\cite{XN} that a Niederreiter-Xing \se with these data is
a $(t,s)$-\se in base $q$ with the same value of $t$ as in Corollary~\ref{co1}. Furthermore, Hofer and Niederreiter~\cite{HN}
proved that the same Niederreiter-Xing \se is also a $(g,\bse,s)$-\se in base $q$ with $\bse =(e_1,\ldots,e_s)$, that is, with
the same parameters as in Theorem~\ref{th1}. Thus, the Niederreiter-Xing \se and our new Halton-type \se have a similar
distribution behavior. However, it should be pointed out that our new \con is considerably simpler than the one in~\cite{XN}.
Since the Niederreiter-Xing \ses and our new Halton-type \ses have the same parameters, the usual procedure of optimizing the
$t$-value for Niederreiter-Xing \ses (see~\cite{NX96} and \cite[Section~8.3]{NX01}) applies in the same way to our new Halton-type
sequences. Consequently, our new \con yields $(t,s)$-\ses in base $q$ such that, for fixed $q$, the parameter $t$ grows linearly
in $s$ as $s \to \infty$ and is therefore asymptotically optimal. }
\end{remark}

\section{Discrepancy bounds} \label{se5}

The star \dis $D_N^*(\cS)$ of any $(t,s)$-\se $\cS$ in any base $b \ge 2$ satisfies an upper bound of the form~\eqref{eqdi}.
In fact, we have
\begin{equation} \label{eq51}
D_N^*(\cS) \le CN^{-1} (\log N)^s + O(N^{-1} (\log N)^{s-1}) \qquad \mbox{for all } N \ge 2
\end{equation}
with the constant $C > 0$ and the implied constant in the Landau symbol depending only on $b$, $s$, and $t$. This was first shown
in \cite[Section~4]{N87} (see also \cite[Theorem 4.17]{N92} for a convenient formulation). The currently best values of $C$ are those
of Faure and Kritzer~\cite{FK}, namely $C=C_{{\rm FK}}$ given by
$$
C_{{\rm FK}}=
\begin{cases}
\frac{b^{t}}{s!} \cdot \frac{b^2}{2(b^2-1)}\left(\frac{b-1}{2\log b}\right)^s & \text{if $b$ is even,} \\
\frac{b^{t}}{s!} \cdot \frac{1}{2}\left(\frac{b-1}{2\log b}\right)^s & \text{if $b$ is odd.}
\end{cases}
$$
It follows that for every base $b \ge 2$ we have
\begin{equation} \label{eq52}
C_{{\rm FK}} \ge \frac{b^t}{s!} \cdot \frac{1}{2} \left(\frac{b-1}{2 \log b} \right)^s.
\end{equation}

Tezuka~\cite{T12} established an upper bound on the star discrepancy
$D_N^*(\cS)$ of any $(u,\bse,s)$-sequence $\cS$ in base $b$. This bound is also of the form~\eqref{eq51} with the constant
$C=C_{{\rm Tez}}$ given by
\begin{equation} \label{eq53}
C_{{\rm Tez}} = \frac{b^u}{s!} \prod_{i=1}^s \frac{\left\lfloor b^{e_i}/2 \right\rfloor}{e_i \log b} \le
\frac{b^u}{s!} \cdot \frac{b^{e_1 + \cdots + e_s}}{(2 \log b)^s e_1 \cdots e_s}.
\end{equation}
Note that our Definition~\ref{de2} of a $(u,\bse,s)$-sequence in base $b$ is slightly stronger than Tezuka's original definition
in~\cite{T12}, and so Tezuka's discrepancy bound is obviously valid for our concept of a $(u,\bse,s)$-sequence in base $b$ as well.

Our new Halton-type sequences are $(t,s)$-sequences in base $q$ as well as $(u,\bse,s)$-sequences
in base $q$ with suitable $t$, $u$, and $\bse$. Therefore we can apply two versions of the discrepancy bound~\eqref{eq51}, namely
the one with $C=C_{{\rm FK}}$ and the one with $C=C_{{\rm Tez}}$. It is of interest to compare these two bounds. Recall
from Theorem~\ref{th1} and Corollary~\ref{co1} that for these \ses
we have $u=g$, the genus of the global function field $F$, as well as $\bse =(e_1,\ldots,e_s)$ and $t=g+ \sum_{i=1}^s
(e_i-1)$, where $e_i=\deg(P_i)$ for $1 \le i \le s$. Then from~\eqref{eq52} and~\eqref{eq53} we obtain
$$
\frac{C_{{\rm FK}}}{C_{{\rm Tez}}} \ge \frac{1}{2} \left(\frac{q-1}{q} \right)^s \prod_{i=1}^s e_i.
$$
Thus, if
\begin{equation} \label{eq54}
\prod_{i=1}^s e_i > 2 \left(\frac{q}{q-1} \right)^s,
\end{equation}
then the discrepancy bound with the constant $C_{{\rm Tez}}$ is better than the one with the constant $C_{{\rm FK}}$. The
condition~\eqref{eq54} will be satisfied in many cases. For instance, it was shown in~\cite{HN} that if we arrange all distinct places
$\ne P_{\infty}$ of $F$ into a list $P_1,P_2,\ldots$ in an arbitrary manner, then for any $0 < \varepsilon < 1/q$ we have
$$
\prod_{i=1}^s \deg(P_i) > (\log_q s)^{((1/q) -\varepsilon)s}
$$
for all sufficiently large $s$, where $\log_q$ denotes the logarithm to the base $q$. Hence in this situation the condition~\eqref{eq54}
is satisfied for all sufficiently large $s$, and so the \dis bound~\eqref{eq51} with the constant $C=C_{{\rm Tez}}$ is better than the
one with the constant $C=C_{{\rm FK}}$.

The discussion above shows that the theory of $(u,\bse,s)$-\ses is a useful extension of the theory of $(t,s)$-\ses since, for instance,
it can lead to improved \dis bounds in comparison with the \dis bounds for $(t,s)$-sequences.

\section*{Acknowledgment}

We are grateful to Professor Chaoping Xing of Nanyang Technological University for suggesting our collaboration and providing valuable
input.

\end{document}